\documentclass[twoside,reqno,A4]{amsart}
\usepackage{latexsym,rotate,eucal,cite}%%%%%%%%%%%%%%%
\usepackage{amsmath,amsthm,amssymb,amsxtra}%%%%%%%%%%%
\usepackage{amsmath}
\usepackage[pagewise]{lineno}
\usepackage{cite}
\newtheorem{theorem}{Theorem}[section]
\newtheorem{lemma}[theorem]{Lemma}
\newtheorem{proposition}[theorem]{Proposition}
\newtheorem{corollary}[theorem]{Corollary}
\newtheorem{remark}[theorem]{Remark}

\theoremstyle{definition}

\allowdisplaybreaks[4]
\numberwithin{equation}{section}

\begin{document}
 \UseRawInputEncoding

\title{On triharmonic hypersurfaces in  space forms }

\author{Yu Fu}

\address[Fu]{School of Data Science and Artificial Intelligence, Dongbei University of Finance and Economics,
Dalian 116025, P. R. China}\email{yufu@dufe.edu.cn}

\author{Dan Yang}
\address[Yang]{School of Mathematics and Statistics, Liaoning University, Shenyang, 110036, China}
\email{dlutyangdan@126.com}


%    General info
\subjclass{Primary 53C40, 58E20; Secondary 53C42}

%\date{January 1, 2001 and, in revised form, June 22, 2001.}


\keywords{$k$-harmonic maps, Triharmonic hypersurfaces,
constant mean curvature, constant scalar curvature}

\begin{abstract}
In this paper we study triharmonic hypersurfaces immersed in a space form $N^{n+1}(c)$.
We prove that any proper CMC triharmonic hypersurface in the sphere $\mathbb S^{n+1}$ has constant scalar curvature; any CMC triharmonic hypersurface in the hyperbolic space $\mathbb H^{n+1}$ is minimal. Moreover, we show that any CMC triharmonic hypersurface in the Euclidean space $\mathbb R^{n+1}$ is minimal provided that the multiplicity of the principal curvature zero is at most one. In particular, we are able to prove that every CMC triharmonic hypersurface in the Euclidean space $\mathbb R^{6}$ is minimal.
These results extend some recent works due to Montaldo-Oniciuc-Ratto and Chen-Guan, and give affirmative answer to the generalized Chen's conjecture.
\end{abstract}

\maketitle \markboth{Fu et al.} {On triharmonic hypersurfaces}

%%%%%%%%%%%%%%%%%%%%%%%%%%%%%%%%%%%%%%%%%%
\section{Introduction}
%\hspace*{\parindent}
%%%%%%%%%%%%%%%%%%%%%%%%%%%%%%%%%%%%%%%%%%%

Let \(\phi :(M,g)\rightarrow (N,{\bar{g}})\) be a smooth map between Riemannian manifolds $M$ and $N$.
A $k$-harmonic map \(\phi \), proposed by Eells and Lemaire \cite{Eells}, is a critical point of the $k$-energy functional

\begin{align*}
E_k(\phi )=\frac{1}{2} \int _M \big |(d+d^{*})^k \phi \big |^2 v_g.
\end{align*}
The Euler-Lagrange equation is given by \(\tau _k(\phi )\equiv 0\) , where \(\tau _k(\phi )\) is the $k$-tension field.
The concept of $k$-harmonic map is a natural generalization of harmonic map, and in particular for \(k=2\) and \(k=3\), the critical points of \(E_2\) or \(E_3\) are biharmonic or triharmonic maps, respectively (c.f.\cite{Wang1989}, \cite{Wang1991},\cite{Maeta2012PAMS}, \cite{Maeta2012Osaka}).

In recent years, biharmonic maps and biharmonic submanifolds have been widely studied, see \cite{Jiang1986, Chen2015, Ou-Chen-book2020} and the references therein. There are also  a lot of results on triharmonic maps and triharmonic submanifolds (c.f.\cite{Maeta2012Osaka, Maeta2012PAMS, Maeta2015houston, Maeta-Naka2015, Nakauchi2018, Montaldo-Ratto20181, Branding2020, Branding2022}). Concerning  triharmonic hypersurfaces in a space form $N^{n+1}(c)$,  Maeta \cite{Maeta2012PAMS} proved that any compact CMC triharmonic hypersurface in $N^{n+1}(c)$ for $c\le 0$ is minimal.  Recently, Montaldo-Oniciuc-Ratto \cite{Montaldo-Oniciuc-Ratto2022} gave a systematic study of CMC triharmonic hypersurface in space forms. The authors proved that

\begin{theorem} {\rm([15])}
Let $M^n$ be a CMC triharmonic hypersurface in $N^{n+1}(c)$ with $c\leq0$ and assume that the squared norm of the second fundamental
form $S$ is constant. Then $M^n$ is minimal.
\end{theorem}

In particular, the assumption $S$ being constant was removed for $n=2$ in \cite{Montaldo-Oniciuc-Ratto2022}.

\begin{theorem} {\rm([15])}
Let $M^2$ be a CMC triharmonic surface in $N^3(c)$.
Then $M^2$ is minimal if $c\leq0$ and $M^2$ is an open part of the small hypersphere  if $c>0$.
\end{theorem}

Very recently, Chen-Guan investigated triharmonic CMC hypersurfaces in a space form $N^{n+1}(c)$ under some assumptions on the number of distinct principal curvatures in \cite{chenguang1, chenguang2}.

\begin{theorem} {\rm([6])}\label{ChenGuan}
Let $M^n$ {\rm($n\geq3$)} be a CMC proper triharmonic hypersurface with
at most three distinct principal curvatures in $N^{n+1}(c)$. Then $M^n$ has constant scalar curvature.
\end{theorem}

\begin{theorem} {\rm([7])}
Let $M^n$ {\rm($n\geq4$)} be a CMC proper triharmonic hypersurface with
four distinct principal curvatures in $N^{n+1}(c)$. If zero is a principal curvature
with multiplicity at most one, then $M^n$ has constant scalar curvature.
\end{theorem}

We recall {\em Chen's conjecture} in the literature of biharmonic submanifolds: {\em any biharmonic submanifold in the Euclidean space $\mathbb R^{n+1}$ is minimal}.
There are some important progress in recent years to support the conjecture under some geometric restrictions (see, for instances \cite{Fu-Yang-Zhan2021, Ou-Chen-book2020}), however the general case remains open. Taking into account Chen's conjecture, Maeta \cite{Maeta2012PAMS} further proposed the generalized Chen's conjecture on $k$-harmonic submanifolds:\\

{\bf Conjecture }: Any $k$-harmonic submanifold in the Euclidean space $\mathbb R^{n+1}$ is minimal.\\

In this paper, we are able to determine the geometry of CMC triharmonic hypersurfaces in a space form $N^{n+1}(c)$ without the restrictions on the number of principal curvatures. We will prove the following statements:

\begin{theorem} \label{mainth1}
Let $M^n$  be a CMC triharmonic
hypersurface in the  hyperbolic space $\mathbb H^{n+1}$. Then $M^n$ is minimal.
\end{theorem}
Moreover, restricting ourselves on the case of $c>0$, we get
\begin{theorem} \label{mainth2}
Let $M^n$  be a CMC proper triharmonic
hypersurface in the sphere $\mathbb S^{n+1}$. Then $M^n$ has constant scalar curvature.
\end{theorem}
Combining our Theorem \ref{mainth2} with the results of Montaldo-Oniciuc-Ratto (Theorem 1.9, \cite{Montaldo-Oniciuc-Ratto2022}) and Chen-Guan (Corollary 1.7, \cite{chenguang2}), we provide a more general result for CMC triharmonic
hypersurfaces in $\mathbb S^{n+1}$.
\begin{corollary}
Let $M^n$  be a CMC proper triharmonic
hypersurface in the sphere $\mathbb S^{n+1}$. Then either
\newline {\rm (1)}\,\,\,$H^2=2$ and $M^n$ is an open part of $S^n(1/\sqrt 3)$, or
\newline {\rm(2)}
$H^2\in(0, t_0]$ and $H^2=t_0$ if and only $M^n$ is an open part of $S^{n-1}(a)\times S^1(\sqrt{1-a^2})$, where $a$ is given by
$$a^2=\frac{2(n-1)^2}{n^2H^2+2n(n-1)+nH\sqrt{n^2H^2+4(n-1)}}$$
and  $t_0$ is the unique real root belonging to $(0, 2)$ of the polynomial
$$f_{n}=n^4t^3-2n^2(n^2-5n+5)t^2-(n-1)(2n-5)(3n-5)t-(n-1)(n-2)^2.$$
\end{corollary}

Let us recall {\em the generalized Chern conjecture} (c.f.\cite{Chang1993,
Chang1994}), which says that:
{\em any closed hypersurface in the unit sphere $\mathbb
S^{n+1}$ with constant mean curvature and constant scalar curvature is isoparametric}. Since the
class of CMC proper triharmonic hypersurfaces in a sphere have constant scalar curvature, the next important problem is to study whether these hypersurfaces are isoparametric. The problem remains open in its full generality. The readers may refer to the recent important progress on the generalized Chern conjecture due to Tang and Yan et al. \cite{Tang-Wei-Yan2020, Tang-Yan2020}.\\

Considering the case $c=0$, we obtain a characterization under an assumption on the multiplicity of zero principal curvature.
\begin{theorem} \label{mainth3}
Let $M^n$  be a CMC triharmonic
hypersurface in the Euclidean space $\mathbb R^{n+1}$. If zero is a principal curvature
with multiplicity at most one, then $M^n$ is minimal.
\end{theorem}
In particular, we can prove
\begin{theorem} \label{mainth4}
Any CMC triharmonic
hypersurface in the Euclidean space $\mathbb R^{6}$ is minimal.
\end{theorem}

\begin{remark}
Note that Theorems \ref{mainth3} and \ref{mainth4}  give partial affirmative answers to the generalized Chen's Conjecture.
\end{remark}

\begin{remark}
The assumption that the multiplicity of the principal curvature  zero is at most one was necessary in {\rm \cite{chenguang2}} for treating triharmonic
hypersurfaces with four distinct principal curvatures in space forms. In our results, we only need this for $c=0$ and $n>5$.
\end{remark}

At last, we point out that for a CMC proper triharmonic
hypersurface in $N^{n+1}(c)$, the two equations in \eqref{triharmonic condition} are quite similar to the equations of a proper biharmonic hypersurface in a space form, see for instance \cite{Fu-Yang-Zhan2021}. This is reasonable because the geometry property of triharmonicity is much weaker than biharmonicity. Hence, it is expected that more geometric features of triharmonic hypersurfaces could be found. Interestingly, we can achieve a complete classification of CMC triharmonic hypersurfaces in $N^{n+1}(c)$ with $c\neq0$. This  will benefit us in studying biharmonic hypersurfaces in $N^{n+1}(c)$.\\

The paper is organized as follows. In Section 2, we recall some
background on the theory of triharmonic hypersurfaces in space forms and derive some useful lemmas, which are very important for us to study the geometric properties of triharmonic hypersurfaces. In Section 3, we give the proofs of  Theorems \ref{mainth1} and  \ref{mainth2}. In Section 4, we finish the proofs of  Theorems \ref{mainth3} and  \ref{mainth4}.

%%%%%%%%%%%%%%%%%%%%%%%%%%%%%%%%%%%%%%%%%%
\section{Preliminaries}

Let $N^{n+1}(c)$ be an $(n+1)$-dimensional Riemannian space form
with constant sectional curvature $c$. For an isometric immersion
$\phi: M^n\rightarrow N^{n+1}(c)$, we denote by $\nabla$ the
Levi-Civita connection of $M^n$ and $\widetilde{\nabla}$  the Levi-Civita
connection of $N^{n+1}(c)$. The Riemannian curvature tensors of $M^n$
are respectively given by
\begin{align*}
&R(X,Y)Z=(\nabla_{X}\nabla_{Y}-\nabla_{Y}\nabla_{X}-\nabla_{[X,Y]})Z,\\
&R(X,Y,Z,W)=\langle R(X,Y)W,Z\rangle.
\end{align*}
The Gauss and Weingarten formulae are stated, respectively, as
\begin{align*}
\widetilde{\nabla}_{X}Y&=\nabla_{X}Y+h(X,Y)\xi,\\
\widetilde{\nabla}_{X}\xi&=-AX.
\end{align*}
Here $X, Y, Z, W$ are tangent vector fields on $M$, $\xi$ is the unit normal vector field on $M$, $h$ is the second fundamental form of $M$, and $A$ is the shape operator.

Let us choose an orthonormal frame $\{e_{i}\}_{i=1}^{n}$ of $M$. With this frame, define $\nabla_{e_{i}}e_{j}=\sum_{k}\Gamma_{ij}^{k}e_{k}$, where $\Gamma_{ij}^k$ are the connection coefficients.

Denote by
\begin{align*}
R_{ijkl}=&R(e_{i},e_{j},e_{k},e_{l}),\quad h_{ij}=h(e_{i},e_{j}),\\
h_{ijk}=&e_{k}(h_{ij})-h(\nabla_{e_{k}}e_{i},e_{j})-h(e_{i},\nabla_{e_{k}}e_{j}),\\
=&e_{k}(h_{ij})-\sum_{l}(\Gamma_{ki}^{l}h_{lj}+\Gamma_{kj}^{l}h_{il}).
\end{align*}
From the definition of the Gauss curvature tensor we obtain
\begin{align}\label{z1}
R_{ijkl}&=e_{i}(\Gamma_{jl}^{k})-e_{j}(\Gamma_{il}^{k})+\sum_{m}\Big(\Gamma_{jl}^{m}\Gamma_{im}^{k}-
\Gamma_{il}^{m}\Gamma_{jm}^{k}-(\Gamma_{ij}^{m}-\Gamma_{ji}^{m})\Gamma_{ml}^{k}\Big).
\end{align}
Moreover, the Gauss and Codazzi equations are given, respectively,  by
\begin{align}
R_{ijkl}&=c(\delta_{ik}\delta_{jl}-\delta_{il}\delta_{jk})+(h_{ik}h_{jl}-h_{il}h_{jk})\label{gauss},\\
h_{ijk}&=h_{ikj}.
\end{align}
The mean curvature function $H$ and the squared norm of the second fundamental form $S$ are written respectively as
\begin{eqnarray}\label{meancur}
H=\frac{1}{n}\sum_{i=1}^n h_{ii} \quad{\rm and}\quad S=\sum_{i,j=1}^n h_{ij}^2.
 \end{eqnarray}
 From the Gauss equation, the scalar curvature $R$ is given by
\begin{eqnarray}\label{scar}
R = n(n-1)c+n^2H^2-S.
\end{eqnarray}
We recall a fundamental characterization result on CMC triharmonic hypersurfaces in $N^{n+1}(c)$.
\begin{proposition} \rm{(c.f.\cite{Montaldo-Oniciuc-Ratto2022})}
A CMC hypersurface $\phi: M^n\rightarrow  N^{n+1}(c)$ is triharmonic
if the mean curvature $H$ and the squared norm of the second fundamental form $S$ on $M^n$
satisfy
\begin{equation}\label{triharmonic condition1}
\left\{
\begin{split}
&H({\rm \Delta}\, S+S^2-ncS-n^2cH^2)=0,\\
&HA\nabla S=0.
\end{split}
\right.
\end{equation}
\end{proposition}

According to \eqref{triharmonic condition1}, it is clear that minimal hypersurfaces are automatically triharmonic in
$N^{n+1}(c)$.  A triharmonic hypersurfaces in $N^{n+1}(c)$ is called {\em proper} if it is not minimal.

In the following, we will consider a CMC proper hypersurface $M^n$
in a space form  $N^{n+1}(c)$. Then \eqref{triharmonic condition1}
becomes
\begin{equation}\label{triharmonic condition}
\left\{
\begin{split}
&{\rm \Delta}\, S+S^2-ncS-n^2cH^2=0,\\
&A\nabla S=0.
\end{split}
\right.
\end{equation}

For a hypersurface $M^n$ in $N^{n+1}(c)$, we denote by $\lambda_i$ for $1\leq i\leq n$ its principal curvatures. The number of distinct principal curvatures is locally constant and  the set of all points here is an open and dense
subset of $M^n$. Denote by $M_A$ this set. On a non-empty connected component
of $M_A$, which is open, the number of distinct principal curvatures is
constant. On that connected component, the multiplicities of the distinct
principal curvatures are constant and hence $\lambda_i$ are always smooth and the shape operator $A$ is
locally diagonalizable (see\cite{Nomizu, Ryan1, Ryan2}).

Denote by $\mathcal{N}:=\{p\in M:\nabla S(p)\neq 0\}$ and $\mathcal{N}\subset M_A$. If $S$ is constant, then $\mathcal{N}$ is an empty set. From
now on, we assume that $S$ is not constant, that is $\mathcal{N}\neq\emptyset$. We will work in
$\mathcal{N}$.

Observing from the second equation of \eqref{triharmonic condition},
it is known that $\nabla S$ is a principal direction with the corresponding principal curvature $0$. Hence, we may choose an orthonormal frame $\{e_{i}\}_{i=1}^{n}$ such that $e_{1}$ is parallel to $\nabla S$ and the shape operator $A$ is diagonalizable with respect to $\{e_{i}\}$, i.e., $h_{ij}=\lambda_{i}\delta_{ij}$, where $\lambda_{i}$ is the principal curvature and $\lambda_{1}=0$.

Suppose that $M^n$ has $d$ distinct principal curvatures $\mu_{1}=0, \mu_{2},\cdots,\mu_{d}$ with $d\geq4$, that is
\begin{equation*}
\lambda_{i}=\mu_{\alpha} \quad {\rm when} \quad i\in I_{\alpha},
\end{equation*}
where
\begin{equation*}
I_{\alpha}=\Big\{\sum_{0\leq\beta\leq\alpha-1}n_{\beta}+1,\cdots,\sum_{0\leq\beta\leq\alpha}n_{\beta}\Big\}
\end{equation*}
with $n_{0}=0$ and $n_{\alpha}\in\mathbb{Z}_{+}$ satisfying $\sum_{1\leq\alpha\leq d}n_{\alpha}=n$, namely, $n_{\alpha}$ is the multiplicity of $\mu_{\alpha}$.
For convenience, we will use the range of indices $1\leq\alpha,\beta,\gamma,\cdots\leq d$ except special declaration.

We collect a lemma for later use.
\begin{lemma}{\rm(c.f.\cite{chenguang2})}\label{lemma1}
The connection coefficients $\Gamma_{ij}^{k}$ satisfy:
\newline {\rm(1)} $\Gamma_{ij}^{k}=-\Gamma_{ik}^{j}$.
\newline {\rm(2)} $\Gamma_{ii}^{k}=\frac{e_{k}(\lambda_{i})}{\lambda_{i}-\lambda_{k}}$ for $i\in I_{\alpha}$ and $k\notin I_{\alpha}$.
\newline {\rm(3)} $\Gamma_{ij}^{k}=\Gamma_{ji}^{k}$ if the indices satisfy one of the following conditions:
\newline \indent {\rm(3a)} $i,j\in I_{\alpha}$ but  $k\notin I_{\alpha}$;
\newline \indent {\rm(3b)} $i,j\geq2$ and $k=1$.
\newline{\rm(4)} $\Gamma_{ij}^{k}=0$ if the indices satisfy one of the following conditions:
\newline\indent {\rm(4a)} $j=k$;
\newline\indent {\rm(4b)} $i=j\in I_{1}$ and $k\notin I_{1}$;
\newline\indent {\rm(4c)} $i,k\in I_{\alpha}, i\neq k$ and $j\notin I_{\alpha}$;
\newline\indent {\rm(4d)} $i,j\geq2, i\in I_{\alpha}, j\in I_{\beta}$ with $\alpha\neq\beta$ and $k=1$.
\newline {\rm(5)} $\Gamma_{ji}^{k}=\frac{\lambda_{j}-\lambda_{k}}{\lambda_{i}-\lambda_{k}}\Gamma_{ij}^{k},$
$\Gamma_{ki}^{j}=\frac{\lambda_{k}-\lambda_{j}}{\lambda_{i}-\lambda_{j}}\Gamma_{ik}^{j}$ for $\lambda_{i}$, $\lambda_{j}$ and $\lambda_{k}$ are mutually different.
\newline{\rm(6)} $\Gamma_{ij}^{k}\Gamma_{ji}^{k}+\Gamma_{ik}^{j}\Gamma_{ki}^{j}+\Gamma_{jk}^{i}\Gamma_{kj}^{i}=0$
for $\lambda_{i}$, $\lambda_{j}$ and $\lambda_{k}$ are mutually different.
\end{lemma}
We first derive some crucial lemmas for studying CMC proper triharmonic hypersurfaces in a space form.
\begin{lemma}\label{lemma2}
Denoting by $P_{\alpha}=\frac{e_{1}(\mu_{\alpha})}{\mu_{\alpha}}$ for $2\leq\alpha\leq d$, we have
\begin{align}
&e_{1}(P_{\alpha})=P_{\alpha}^{2}+c+\sum_{m\in I_{1}}\Gamma_{\alpha\alpha}^{m}\Gamma_{11}^{m},\label{F1} \\
&e_{j}(P_{\alpha})=\Gamma_{\alpha\alpha}^{j}P_{\alpha}-\sum_{m\in I_{1}}\Gamma_{\alpha\alpha}^{m}\Gamma_{jm}^{1} \,\,\,{\rm for}\,\, j\in I_{1} \,\,{\rm and}\,\, j\neq1.\label{F2}
\end{align}
\end{lemma}

\begin{proof}

For $i\in I_{\alpha}$ and $\alpha\neq1$, it follows from  \eqref{z1} and the terms (1), (2) and (4) in Lemma \ref{lemma1} that
\begin{align*}
R_{1i1i}=&e_{1}(\Gamma_{ii}^{1})-e_{i}(\Gamma_{1i}^{1})+\sum_{m}\Big(\Gamma_{ii}^{m}\Gamma_{1m}^{1}-
\Gamma_{1i}^{m}\Gamma_{im}^{1}-(\Gamma_{1i}^{m}-\Gamma_{i1}^{m})\Gamma_{mi}^{1}\Big)\\
=&e_{1}(\Gamma_{ii}^{1})+e_{i}(\Gamma_{11}^{i})-\sum_{m}\Big(\Gamma_{ii}^{m}\Gamma_{11}^{m}+
\Gamma_{1i}^{m}\Gamma_{im}^{1}+(\Gamma_{1i}^{m}+\Gamma_{im}^{1})\Gamma_{mi}^{1}\Big)\\
=&e_{1}(\Gamma_{ii}^{1})-\sum_{m\in I_{1}}\Big(\Gamma_{ii}^{m}\Gamma_{11}^{m}+
\Gamma_{1i}^{m}\Gamma_{im}^{1}+(\Gamma_{1i}^{m}+\Gamma_{im}^{1})\Gamma_{mi}^{1}\Big)\\
=&e_{1}(\Gamma_{ii}^{1})-(\Gamma_{ii}^{1})^{2}-\sum_{m\in I_{1}}\Gamma_{ii}^{m}\Gamma_{11}^{m}.
\end{align*}
On the other hand, from the Gauss equation \eqref{gauss} we get $R_{1i1i}=c$. Therefore, we obtain  \eqref{F1}.

For $i\in I_{\alpha}$ and $\alpha\neq1$, $j\in I_{1}$ and $j\neq1$, it follows from \eqref{z1} and Lemma \ref{lemma1} that
\begin{align}
R_{iji1}&=e_{i}(\Gamma_{j1}^{i})-e_{j}(\Gamma_{i1}^{i})+\sum_{m}\Big(\Gamma_{j1}^{m}\Gamma_{im}^{i}-
\Gamma_{i1}^{m}\Gamma_{jm}^{i}-(\Gamma_{ij}^{m}-\Gamma_{ji}^{m})\Gamma_{m1}^{i}\Big)\nonumber\\
&=-e_{i}(\Gamma_{ji}^{1})+e_{j}(\Gamma_{ii}^{1})+\sum_{m}\Big(\Gamma_{jm}^{1}\Gamma_{ii}^{m}+
\Gamma_{im}^{1}\Gamma_{jm}^{i}+(\Gamma_{ij}^{m}-\Gamma_{ji}^{m})\Gamma_{mi}^{1}\Big)\nonumber\\
&=e_{j}(\Gamma_{ii}^{1})+\sum_{m\in I_1}\Gamma_{jm}^{1}\Gamma_{ii}^{m}-\Gamma_{ii}^{j}\Gamma_{ii}^{1},
\end{align}
which together with $R_{iji1}=0$ gives \eqref{F2}. Note that  $R_{iji1}=0$ follows from the Gauss equation \eqref{gauss} directly.
We thus complete the proof.
\end{proof}
\begin{lemma}\label{lemma3}
For any $q\in \mathbb{Z}_{+}$,  we have
\begin{align}\label{t3}
\begin{split}
f(q):=\sum_{2\leq\alpha\leq d}n_{\alpha}\mu_{\alpha}P^{q}_{\alpha}=
\left \{
\begin{array}{ll}
    \quad 0,                    & {\rm when} \,\,q \,\, {\rm is} \,\, {\rm odd};\\
       \frac{(q-1)!!}{q!!}(-c)^{\frac{q}{2}}nH,     & {\rm when} \,\, q \,\,  {\rm is} \,\, {\rm even}.
\end{array}
\right.
\end{split}
\end{align}
 \end{lemma}
\begin{proof}
Since the case $n_1=1$ has been obtained in \cite{chenguang2}, we only need to prove it for $n_1>1$.

Taking into account the definition of $H$ from the first expression of \eqref{meancur}, we have
\begin{align}\label{mean2}
\sum_{2\leq\alpha\leq d}n_{\alpha}\mu_{\alpha}=nH.
\end{align}
Since $H$ is constant, differentiating \eqref{mean2} with respect to $e_{1}$, we obtain
\begin{align*}
\sum_{2\leq\alpha\leq d}n_{\alpha}e_{1}(\mu_{\alpha})=0,
\end{align*}
which is equivalent to
\begin{align}\label{t1}
\sum_{2\leq\alpha\leq d}n_{\alpha}\mu_{\alpha}P_{\alpha}=0.
\end{align}
Differentiating \eqref{mean2} with respect to $e_{m}$ for $m\in I_1$ and $m\neq1$, we obtain
\begin{align}\label{FF1}
\sum_{2\leq\alpha\leq d}n_{\alpha}e_{m}(\mu_{\alpha})=0.
\end{align}
Differentiating \eqref{t1} with respect to $e_{1}$, from \eqref{F1}, \eqref{mean2}, $\mu_{\alpha}\Gamma_{\alpha\alpha}^{m}=e_m(\mu_{\alpha})$ and \eqref{FF1} we have
\begin{align*}
0&=\sum_{2\leq\alpha\leq d}n_{\alpha}\Big(e_{1}(\mu_{\alpha})P_{\alpha}+\mu_{\alpha}e_{1}(P_{\alpha})\Big)\nonumber\\
&=\sum_{2\leq\alpha\leq d}n_{\alpha}\Big(\mu_{\alpha}P_{\alpha}^{2}+\mu_{\alpha}(P_{\alpha}^{2}+c+\sum_{m\in I_{1}}\Gamma_{\alpha\alpha}^{m}\Gamma_{11}^{m})\Big)\nonumber\\
&=2\sum_{2\leq\alpha\leq d}n_{\alpha}\mu_{\alpha}P_{\alpha}^{2}+cnH+\sum_{2\leq\alpha\leq d}\sum_{m\in I_{1}}n_{\alpha}\mu_{\alpha}\Gamma_{\alpha\alpha}^{m}\Gamma_{11}^{m}\nonumber\\
&=2\sum_{2\leq\alpha\leq d}n_{\alpha}\mu_{\alpha}P_{\alpha}^{2}+cnH+\sum_{m\in I_{1}}\Big(\sum_{2\leq\alpha\leq d}n_{\alpha}e_m(\mu_{\alpha})\Big)\Gamma_{11}^{m}\nonumber\\
&=2\sum_{2\leq\alpha\leq d}n_{\alpha}\mu_{\alpha}P_{\alpha}^{2}+cnH,
\end{align*}
which is equivalent to
\begin{align}\label{t2}
\sum_{2\leq\alpha\leq d}n_{\alpha}\mu_{\alpha}P_{\alpha}^{2}=-\frac{1}{2}cnH.
\end{align}
Equations \eqref{t1} and \eqref{t2} imply that \eqref{t3} holds for $q=1,2$. Next we will prove that it holds for general $q$ by induction.

Differentiating \eqref{t3} with respect to $e_{1}$ yields
\begin{align}\label{t7}
0&=\sum_{2\leq\alpha\leq d}n_{\alpha}\Big(e_{1}(\mu_{\alpha})P_{\alpha}^{q}+q\mu_{\alpha}P_{\alpha}^{q-1}e_{1}(P_{\alpha})\Big)\nonumber\\
&=\sum_{2\leq\alpha\leq d}n_{\alpha}\Big(\mu_{\alpha}P_{\alpha}^{q+1}+q\mu_{\alpha}P_{\alpha}^{q-1}(P_{\alpha}^{2}+c+\sum_{m\in I_{1}}\Gamma_{\alpha\alpha}^{m}\Gamma_{11}^{m})\Big)\nonumber\\
&=\sum_{2\leq\alpha\leq d}\Big((1+q)n_{\alpha}\mu_{\alpha}P_{\alpha}^{q+1}+cqn_{\alpha}\mu_{\alpha}P_{\alpha}^{q-1}\Big)+
q\sum_{2\leq\alpha\leq d}n_{\alpha}P_{\alpha}^{q-1}\sum_{m\in I_{1}}e_{m}(\mu_{\alpha})\Gamma_{11}^{m}\nonumber\\
&=(1+q)f(q+1)+cqf(q-1)+q\sum_{m\in I_{1}}\Big\{\sum_{2\leq\alpha\leq d}n_{\alpha}e_{m}(\mu_{\alpha})P_{\alpha}^{q-1}\Big\}\Gamma_{11}^{m}.
\end{align}
On the other hand, since  \eqref{t3} holds for $f(q-1)$,  we differentiate $f(q-1)=\sum_{2\leq\alpha\leq d}n_{\alpha}\mu_{\alpha}P^{q-1}_{\alpha}$ with respect to $e_{j}$ for $j\in I_1$ and $j\neq1$.
It follows from \eqref{F2} that
\begin{align}
0&=\sum_{2\leq\alpha\leq d}\Big(n_{\alpha}e_{j}(\mu_{\alpha})P_{\alpha}^{q-1}+(q-1)n_{\alpha}\mu_{\alpha}P_{\alpha}^{q-2}e_{j}(P_{\alpha})\Big)\nonumber\\
&=\sum_{2\leq\alpha\leq d}\Big(n_{\alpha}e_{j}(\mu_{\alpha})P_{\alpha}^{q-1}+(q-1)n_{\alpha}\mu_{\alpha}P_{\alpha}^{q-2}\big(\Gamma_{\alpha\alpha}^{j}P_{\alpha}-\sum_{m\in I_1}\Gamma_{\alpha\alpha}^{m}\Gamma_{jm}^1)\Big)\nonumber\\
&=\sum_{2\leq\alpha\leq d}\Big(n_{\alpha}e_{j}(\mu_{\alpha})P_{\alpha}^{q-1}+(q-1)n_{\alpha}e_{j}(\mu_{\alpha})P_{\alpha}^{q-1}-(q-1)\sum_{m\in I_1}n_{\alpha}P_{\alpha}^{q-2}e_{m}(\mu_{\alpha})\Gamma_{jm}^{1}\Big)\nonumber\\
&=\sum_{2\leq\alpha\leq d}\Big(q n_{\alpha}e_{j}(\mu_{\alpha})P_{\alpha}^{q-1}-(q-1)\sum_{m\in I_1}n_{\alpha}P_{\alpha}^{q-2}e_{m}(\mu_{\alpha})\Gamma_{jm}^{1}\Big).\nonumber
\end{align}
Hence the following relation holds for any $j\in I_1$ and $j\neq 1$
\begin{align}\label{T1}
q\sum_{2\leq\alpha\leq d}n_{\alpha}e_{j}(\mu_{\alpha})P_{\alpha}^{q-1}=\sum_{m\in I_1}\Big\{(q-1)\sum_{2\leq\alpha\leq d}n_{\alpha}e_{m}(\mu_{\alpha})P_{\alpha}^{q-2}\Big\}\Gamma_{jm}^{1}.
\end{align}
Since  \eqref{T1} holds for any $q$,  letting $q=2$, \eqref{T1} reduces to
\begin{align*}
2\sum_{2\leq\alpha\leq d}n_{\alpha}e_{j}(\mu_{\alpha})P_{\alpha}=\sum_{m\in I_1}\Big\{\sum_{2\leq\alpha\leq d}n_{\alpha}e_{m}(\mu_{\alpha})\Big\}\Gamma_{jm}^{1},
\end{align*}
which together with \eqref{FF1} yields
\begin{align}\label{TT1}
2\sum_{2\leq\alpha\leq d}n_{\alpha}e_{j}(\mu_{\alpha})P_{\alpha}=0.
\end{align}
Letting $q=3$ and using \eqref{TT1}, \eqref{T1} reduces to
\begin{align*}
3\sum_{2\leq\alpha\leq d}n_{\alpha}e_{j}(\mu_{\alpha})P_{\alpha}^2=\sum_{m\in I_1}\Big\{2\sum_{2\leq\alpha\leq d}n_{\alpha}e_{m}(\mu_{\alpha})P_{\alpha}\Big\}\Gamma_{jm}^{1}=0.
\end{align*}
Similarly, we can gradually show
\begin{align}\label{t8}
q\sum_{2\leq\alpha\leq d}n_{\alpha}e_{j}(\mu_{\alpha})P_{\alpha}^{q-1}=0.
\end{align}
Hence, combing \eqref{t7} with \eqref{t8} gives
\begin{align}\label{t9}
(q+1)f(q+1)+cqf(q-1)=0.
\end{align}

When $q$ is even, both of $q-1$ and $q+1$ are odd.  From \eqref{t9}, $f(q-1)=0$ can yield $f(q+1)=0$ as well.

When $q$ is odd, both of $q-1$ and $q+1$ are even. We conclude from \eqref{t9} that
\begin{align}\label{t10}
f(q+1)&=-\frac{cq}{q+1}f(q-1)\nonumber\\
&=-\frac{cq}{q+1}\times\frac{(q-2)!!}{(q-1)!!}(-c)^{(q-1)/2}nH\nonumber\\
&=\frac{q!!}{(q+1)!!}(-c)^{(q+1)/2}nH,
\end{align}
which completes the proof of Lemma \ref{lemma3}.
\end{proof}
\begin{remark}
While assuming an additional condition that the multiplicity of zero principal curvature is at most one, special cases of Lemmas \ref{lemma2} and \ref{lemma3} were derived in \cite{chenguang2}.
\end{remark}
\begin{lemma}\label{lemma4}
The equation ${\rm \Delta}\, S+S^2-ncS-n^2cH^2=0$  is equivalent to
\begin{align*}
&-6\sum_{\alpha=2}^{d}n_{\alpha}\mu_{\alpha}^{2}P_{\alpha}^{2}
+2\Big(\sum_{\alpha=2}^{d}n_{\alpha}\mu_{\alpha}^{2}P_{\alpha}\Big)\Big(\sum_{\alpha=2}^{d}n_{\alpha}P_{\alpha}+\sum_{m\in I_{1}}\Gamma_{mm}^{1}\Big)
+\Big(\sum_{\alpha=2}^{d}n_{\alpha}\mu_{\alpha}^{2}\Big)^{2}\\&-(n+2)c\Big(\sum_{\alpha=2}^{d}n_{\alpha}\mu_{\alpha}^{2}\Big)
-c\Big(\sum_{\alpha=2}^{d}n_{\alpha}\mu_{\alpha}\Big)^{2}=0.
\end{align*}
 \end{lemma}
 \begin{proof}
Since $e_i(S)=0$ for $2\leq i\leq n$, it follows from Lemma \ref{lemma1} that
\begin{align}\label{eqsystem6}
\Delta S&=-\sum_{i=1}^{n}(\nabla_{e_{i}}\nabla_{e_{i}}S-
\nabla_{\nabla_{e_{i}e_{i}}}S)\nonumber\\
&=-e_{1}e_{1}S+e_{1}(S)\sum_{i=2}^{n}\Gamma_{ii}^{1}\nonumber\\
&=-e_{1}e_{1}S+e_{1}(S)\Big(\sum_{\alpha=2}^{d}n_{\alpha}P_{\alpha}+\sum_{m\in I_{1}}\Gamma_{mm}^{1}\Big).
\end{align}
Noting $S=\sum_{\alpha=2}^{d}n_{\alpha}\mu_{\alpha}^{2}$, it follows from \eqref{F1} that
\begin{align}\label{eqsystem7}
e_{1}(S)&=2\sum_{\alpha=2}^{d}n_{\alpha}\mu_{\alpha}e_{1}(\mu_{\alpha})=2\sum_{\alpha=2}^{d}n_{\alpha}\mu_{\alpha}^{2}P_{\alpha},\\
e_{1}e_{1}(S)&=4\sum_{\alpha=2}^{d}n_{\alpha}\mu_{\alpha}^{2}P_{\alpha}^{2}+
2\sum_{\alpha=2}^{d}n_{\alpha}\mu_{\alpha}^{2}P_{\alpha}^{2}+2c\sum_{\alpha=2}^{d}n_{\alpha}\mu_{\alpha}^{2}+2\sum_{\alpha=2}^{d}\sum_{m\in I_{1}}n_{\alpha}\mu_{\alpha}^{2}\Gamma_{\alpha\alpha}^{m}\Gamma_{11}^{m}\nonumber\\
&=
6\sum_{\alpha=2}^{d}n_{\alpha}\mu_{\alpha}^{2}P_{\alpha}^{2}+2c\sum_{\alpha=2}^{d}n_{\alpha}\mu_{\alpha}^{2}+2\sum_{m\in I_{1}}\Big\{\sum_{\alpha=2}^{d}n_{\alpha}\mu_{\alpha}^{2}\Gamma_{\alpha\alpha}^{m}\Big\}\Gamma_{11}^{m}.\label{FFF0}
\end{align}
Differentiating $S=\sum_{\alpha=2}^{d}n_{\alpha}\mu_{\alpha}^{2}$ with respect to $e_m$ for $m\in I_1$ and $m\neq1$, we get
\begin{align}\label{FFF1}
0=e_m(S)=2\sum_{\alpha=2}^{d}n_{\alpha}\mu_{\alpha}e_m(\mu_{\alpha})=2\sum_{\alpha=2}^{d}n_{\alpha}\mu_{\alpha}^2\Gamma_{\alpha\alpha}^{m}.
\end{align}
Combining \eqref{FFF0} with \eqref{FFF1} gives
\begin{align}\label{FFF2}
e_{1}e_{1}(S)=6\sum_{\alpha=2}^{d}n_{\alpha}\mu_{\alpha}^{2}P_{\alpha}^{2}+2c\sum_{\alpha=2}^{d}n_{\alpha}\mu_{\alpha}^{2}.
\end{align}
Substituting \eqref{eqsystem7} and \eqref{FFF2} into  \eqref{eqsystem6}, we have
\begin{align*}
\Delta S=&-6\sum_{\alpha=2}^{d}n_{\alpha}\mu_{\alpha}^{2}P_{\alpha}^{2}+
2\Big(\sum_{\alpha=2}^{d}n_{\alpha}\mu_{\alpha}^{2}P_{\alpha}\Big)\Big(\sum_{\alpha=2}^dn_{\alpha}P_{\alpha}+\sum_{m\in I_{1}}\Gamma_{mm}^{1}\Big)\nonumber\\
&-2c\sum_{\alpha=2}^{d}n_{\alpha}\mu_{\alpha}^{2}.
\end{align*}
Hence, the proof has been completed.
\end{proof}

%%%%%%%%%%%%%%%%%%%%%%%%%%%%%%%%%%%%%%%%%%
\section{Proof of Theorems \ref{mainth1} and \ref{mainth2}}

{\bf The proof of Theorems \ref{mainth1} and \ref{mainth2}}:
We will prove Theorems \ref{mainth1} and \ref{mainth2} by deriving  a contradiction from the assumption that
$\mathcal{N}=\{p\in M^n:\nabla S(p)\neq 0\}\neq\emptyset$.

Taking $q=1,3,5,\cdots,2d-3$ in Lemma \ref{lemma3}, we have
\begin{align}\label{eqsystem0}
\begin{split}
\left \{
\begin{array}{ll}
    n_{2}\mu_{2}P_{2}+n_{3}\mu_{3}P_{3}+\cdots+n_{d}\mu_{d}P_{d}=0,\\
    n_{2}\mu_{2}P_{2}^{3}+n_{3}\mu_{3}P_{3}^{3}+\cdots+n_{d}\mu_{d}P_{d}^{3}=0,\\
\quad \,\vdots\\
   n_{2}\mu_{2}P_{2}^{2d-3}+n_{3}\mu_{3}P_{3}^{2d-3}+\cdots+n_{d}\mu_{d}P_{d}^{2d-3}=0,
\end{array}
\right.
\end{split}
\end{align}
which is a $(d-1)$-th order equation system with a non-zero solution. Hence on $\mathcal{N}$ we have
\begin{align}\label{det1}
\begin{split}
\left |
\begin{array}{llll}
    P_{2} &P_{3} &\cdots &P_{d}\\
    P_{2}^{3} &P_{3}^{3} &\cdots &P_{d}^{3}\\
\,\,\vdots&\,\,\vdots&\cdots&\,\,\vdots\\
  P_{2}^{2d-3} &P_{3}^{2d-3}&\cdots &P_{d}^{2d-3}
\end{array}
\right |=P_{2}\cdots P_{d}\prod_{2\leq\alpha\leq\beta\leq d}(P_{\alpha}^{2}-P_{\beta}^{2})=0.
\end{split}
\end{align}

Taking $q=2,4,6,\cdots,2d-2$ in Lemma \ref{lemma3}, we obtain
\begin{align}\label{eqsystem3}
\begin{split}
\left \{
\begin{array}{ll}
    n_{2}\mu_{2}P_{2}^{2}+n_{3}\mu_{3}P_{3}^{2}+\cdots+n_{d}\mu_{d}P_{d}^{2}=-\frac{1}{2}ncH,\\
    n_{2}\mu_{2}P_{2}^{4}+n_{3}\mu_{3}P_{3}^{4}+\cdots+n_{d}\mu_{d}P_{d}^{4}=\frac{3}{8}nc^{2}H,\\
\quad\,\vdots\\
   n_{2}\mu_{2}P_{2}^{2d-2}+n_{3}\mu_{3}P_{3}^{2d-2}+\cdots+n_{d}\mu_{d}P_{d}^{2d-2}=\frac{(2d-3)!!}{(2d-2)!!}(-c)^{d-1}nH.
\end{array}
\right.
\end{split}
\end{align}

Next we consider all possible cases.

{\bf Case 1.} $P_{2}P_{3}\cdots P_{d}\neq0$ at some $p\in\mathcal{N}$. Then from \eqref{det1} we have that
$\prod_{2\leq\alpha\leq\beta\leq d}(P_{\alpha}^{2}-P_{\beta}^{2})=0$ at $p$.
Without loss of generality, we assume $P_{2}^{2}-P_{3}^{2}=0$, i.e., $P_{2}=\pm P_{3}$.
 Now the former $d-2$ equations in the system of equations \eqref{eqsystem0} determine a new system of equations as follows:
\begin{align}\label{eqsystem2}
\begin{split}
\left \{
\begin{array}{ll}
   (n_{2}\mu_{2}\pm n_{3}\mu_{3})P_{3}+\cdots+n_{d}\mu_{d}P_{d}=0,\\
   (n_{2}\mu_{2}\pm n_{3}\mu_{3})P_{3}^{3}+\cdots+n_{d}\mu_{d}P_{d}^{3}=0,\\
\quad\quad\quad\vdots\\
  (n_{2}\mu_{2}\pm n_{3}\mu_{3})P_{3}^{2d-5}+\cdots+n_{d}\mu_{d}P_{d}^{2d-5}=0,
\end{array}
\right.
\end{split}
\end{align}
which has a non-zero solution. Then we have
\begin{align*}
\begin{split}
\left |
\begin{array}{llll}
   P_{3} &P_{4}& \cdots &P_{d}\\
     P_{3}^{3}&P_{4}^{3}&\cdots &P_{d}^{3}\\
\,\,\vdots&\,\,\vdots&\cdots&\,\,\vdots\\
  P_{3}^{2d-5} &P_{4}^{2d-5} & \cdots &P_{d}^{2d-5}
\end{array}
\right |=P_{3}\cdots P_{d}\prod_{3\leq\alpha\leq\beta\leq d}(P_{\alpha}^{2}-P_{\beta}^{2})=0.
\end{split}
\end{align*}
Since $P_{3}P_4\cdots P_{d}\neq0$, we have that $\prod_{3\leq\alpha\leq\beta\leq d}(P_{\alpha}^{2}-P_{\beta}^{2})=0$.
Without loss of generality,
we assume that $P_{3}^{2}=P_{4}^{2}$.
Proceeding in this way, we obtain that
$P_{2}^{2}=P_{3}^{2}=\cdots=P_{d}^{2}:=P^2$ at $p$.
 Now \eqref{eqsystem3} becomes
\begin{align*}
\begin{split}
\left \{
\begin{array}{ll}
    -\frac{1}{2}ncH=n_{2}\mu_{2}P_{2}^{2}+n_{3}\mu_{3}P_{3}^{2}+\cdots+n_{d}\mu_{d}P_{d}^{2}=nHP^{2},\\
    \frac{3}{8}nc^{2}H=n_{2}\mu_{2}P_{2}^{4}+n_{3}\mu_{3}P_{3}^{4}+\cdots+n_{d}\mu_{d}P_{d}^{4}=nHP^{4},\\
\quad\,\,\vdots\\
   \frac{(2d-3)!!}{(2d-2)!!}(-c)^{d-1}nH=n_{2}\mu_{2}P_{2}^{2d-2}+n_{3}\mu_{3}P_{3}^{2d-2}+\cdots+n_{d}\mu_{d}P_{d}^{2d-2}=nHP^{2d-2}.
\end{array}
\right.
\end{split}
\end{align*}
Since $nH\neq0$, the first two equations of the above system  imply that
$$P^2=-\frac{1}{2}c \quad {\rm and}\quad P^4=\frac{3}{8}c^2,$$
and hence $c=P=0$. It is a contradiction, so this case is ruled out.

{\bf Case 2.} $P_{\alpha}=0$ for all $\alpha=2,\cdots,d$ at some $p\in\mathcal{N}$. In this case, we have
$$e_{1}S=e_{1}\sum_{2\leq\alpha\leq d}n_{\alpha}\mu_{\alpha}^{2}=2\sum_{2\leq\alpha\leq d}n_{\alpha}\mu_{\alpha}e_{1}(\mu_{\alpha})=2\sum_{2\leq\alpha\leq d}n_{\alpha}\mu_{\alpha}^{2}P_{\alpha}=0.$$
 This contradicts $\nabla S\neq0$ at $p$ and this case is also ruled out.

{\bf Case 3.} For any given point $p\in\mathcal{N}$, some terms of  $P_{\alpha}$ are zero and the others are not zero. In this case, without loss of generality,
assume $P_{\alpha}=0$ for $\alpha=2,\cdots, r$ and $P_{\alpha}\neq 0$ for $\alpha=r+1,\cdots,d$. Then the first $d-r$ equations in \eqref{eqsystem0} form a new system of equations
\begin{align}\label{eqsystem}
\begin{split}
\left \{
\begin{array}{ll}
    n_{r+1}\mu_{r+1}P_{r+1}+n_{r+2}\mu_{r+2}P_{r+2}+\cdots+n_{d}\mu_{d}P_{d}=0,\\
    n_{r+1}\mu_{r+1}P_{r+1}^{3}+n_{r+2}\mu_{r+2}P_{r+2}^{3}+\cdots+n_{d}\mu_{d}P_{d}^{3}=0,\\
\quad\quad\quad\vdots\\
   n_{r+1}\mu_{r+1}P_{r+1}^{2(d-r)-1}+n_{r+2}\mu_{r+2}P_{r+2}^{2(d-r)-1}+\cdots+n_{d}\mu_{d}P_{d}^{2(d-r)-1}=0,
\end{array}
\right.
\end{split}
\end{align}
which is a $(d-r)$-th order equation system with non-zero solutions. So the coefficient determinant is zero, that is $\prod_{r+1\leq\alpha\leq\beta\leq d}(P_{\alpha}^{2}-P_{\beta}^{2})=0$. Without loss of generality,
we assume that $P_{r+1}^{2}=P_{r+2}^{2}$.  Proceeding in this way, we can show that
$P_{r+1}^{2}=\cdots=P_{d}^{2}\neq0$.  Denote by $P^2:=P_{r+1}^{2}=\cdots=P_{d}^{2}$. Then  \eqref{eqsystem3} becomes
\begin{align}\label{eqsystem16}
\begin{split}
\left \{
\begin{array}{ll}
   (n_{r+1}\mu_{r+1}+\cdots+n_{d}\mu_{d})P^{2}=-\frac{1}{2}ncH,\\
  (n_{r+1}\mu_{r+1}+\cdots+n_{d}\mu_{d})P^{4}=\frac{3}{8}nc^{2}H,\\
    (n_{r+1}\mu_{r+1}+\cdots+n_{d}\mu_{d})P^{6}=-\frac{5}{16}nc^{3}H,\\
\quad\quad\quad\vdots\\
(n_{r+1}\mu_{r+1}+\cdots+n_{d}\mu_{d})P^{2d-2}=\frac{(2d-3)!!}{(2d-2)!!}(-c)^{d-1}nH.
\end{array}
\right.
\end{split}
\end{align}
The above system of equations means that $n_{r+1}\mu_{r+1}+\cdots+n_{d}\mu_{d}\neq0$ since $c\neq0$ and $H\neq0$. The first two equations of  \eqref{eqsystem16} force that $P^{2}=-\frac{3}{4}c$, and the second and the third equation of \eqref{eqsystem16} force that $P^{2}=-\frac{5}{6}c$. Hence we have $P=c=0$, a contradiction.

In conclusion, we have that $\mathcal{N}=\{p\in M:\nabla S(p)\neq 0\}$ is empty and hence $S$ has to be a constant.
From \eqref{scar}, we conclude that the scalar curvature $R$ of $M^n$ is constant as well. But for $c<0$, the first equation of \eqref{triharmonic condition} means that $S^2-ncS-n^2cH^2=0$, a contradiction. This completes the proof of Theorems \ref{mainth1} and \ref{mainth2}. $\hfill\square$

%%%%%%%%%%%%%%%%%%%%%%%%%%%%%%%%%%%%%%%%%%
\section{Proofs of Theorems \ref{mainth3} and \ref{mainth4}}

In this section, we mainly concern CMC triharmonic hypersurfaces in the Euclidean space $\mathbb R^{n+1}$. For any dimension $n$, we need another assumption that the multiplicity of  the zero principal curvature is at most one as discussed in \cite{chenguang2} for four distinct principal curvatures.

{\bf The proof of Theorem \ref{mainth3}}:
Assume that $\mathcal{N}=\{p\in M^n:\nabla S(p)\neq 0\}\neq\emptyset$. We will prove Theorems \ref{mainth3} by deriving  a contradiction.

Taking $q=1,2,3,\cdots,d-1$ in Lemma \ref{lemma3}, we have
\begin{align}\label{eqsystem04}
\begin{split}
\left \{
\begin{array}{ll}
    n_{2}\mu_{2}P_{2}+n_{3}\mu_{3}P_{3}+\cdots+n_{d}\mu_{d}P_{d}=0,\\
    n_{2}\mu_{2}P_{2}^{2}+n_{3}\mu_{3}P_{3}^{2}+\cdots+n_{d}\mu_{d}P_{d}^{2}=0,\\
\quad \,\vdots\\
   n_{2}\mu_{2}P_{2}^{d-1}+n_{3}\mu_{3}P_{3}^{d-1}+\cdots+n_{d}\mu_{d}P_{d}^{d-1}=0,
\end{array}
\right.
\end{split}
\end{align}
which is a $(d-1)$-th order equation system with a non-zero solution. Hence on $\mathcal{N}$ we have
\begin{align}\label{det01}
\begin{split}
\left |
\begin{array}{llll}
    P_{2} &P_{3} &\cdots &P_{d}\\
    P_{2}^{2} &P_{3}^{2} &\cdots &P_{d}^{2}\\
\,\,\vdots&\,\,\vdots&\cdots&\,\,\vdots\\
  P_{2}^{d-1} &P_{3}^{d-1}&\cdots &P_{d}^{d-1}
\end{array}
\right |=P_{2}\cdots P_{d}\prod_{2\leq\alpha\leq\beta\leq d}(P_{\alpha}-P_{\beta})=0.
\end{split}
\end{align}

Next we consider three possible cases.

{\bf Case 1.} $P_{2}P_{3}\cdots P_{d}\neq0$ at some $p\in\mathcal{N}$. Then from \eqref{det01} we have that
$\prod_{2\leq\alpha\leq\beta\leq d}(P_{\alpha}-P_{\beta})=0$ at $p$.
Without loss of generality, we assume $P_{2}-P_{3}=0$.
 Now the former $d-2$ equations in the system of equations \eqref{eqsystem04} determine a new system of equations as follows:
\begin{align}\label{eqsystem02}
\begin{split}
\left \{
\begin{array}{ll}
   (n_{2}\mu_{2}+ n_{3}\mu_{3})P_{3}+\cdots+n_{d}\mu_{d}P_{d}=0,\\
   (n_{2}\mu_{2}+ n_{3}\mu_{3})P_{3}^{2}+\cdots+n_{d}\mu_{d}P_{d}^{2}=0,\\
\quad\quad\quad\vdots\\
  (n_{2}\mu_{2}+ n_{3}\mu_{3})P_{3}^{d-2}+\cdots+n_{d}\mu_{d}P_{d}^{d-2}=0,
\end{array}
\right.
\end{split}
\end{align}
which has a non-zero solution. Then we have
\begin{align*}
\begin{split}
\left |
\begin{array}{llll}
   P_{3} &P_{4}& \cdots &P_{d}\\
     P_{3}^{2}&P_{4}^{2}&\cdots &P_{d}^{2}\\
\,\,\vdots&\,\,\vdots&\cdots&\,\,\vdots\\
  P_{3}^{d-2} &P_{4}^{d-2} & \cdots &P_{d}^{d-2}
\end{array}
\right |=P_{3}\cdots P_{d}\prod_{3\leq\alpha\leq\beta\leq d}(P_{\alpha}-P_{\beta})=0.
\end{split}
\end{align*}
Since $P_{3}P_{4}\cdots P_{d}\neq0$, we have that $\prod_{3\leq\alpha\leq\beta\leq d}(P_{\alpha}-P_{\beta})=0$.
Without loss of generality,
we assume that $P_{3}=P_{4}$.
Similarly, we obtain that
$P_{3}=P_{4}=\cdots=P_{d}:=P$ at $p$.
Since $nH\neq0$, the first equation of the above system \eqref{eqsystem02} implies that
$$\Big(\sum_{2\leq\alpha\leq d}{n_{\alpha}\mu_{\alpha}}\Big)P=nHP=0$$
and hence $P=0$. It is a contradiction.

{\bf Case 2.} $P_{\alpha}=0$ for all $\alpha=2,\cdots,d$ at some $p\in\mathcal{N}$. Then
$$e_{1}S=e_{1}\sum_{2\leq\alpha\leq d}n_{\alpha}\mu_{\alpha}^{2}=2\sum_{2\leq\alpha\leq d}n_{\alpha}\mu_{\alpha}e_{1}(\mu_{\alpha})=2\sum_{2\leq\alpha\leq d}n_{\alpha}\mu_{\alpha}^{2}P_{\alpha}=0,$$ which contradicts $\nabla S\neq0$ at $p$.

{\bf Case 3.} For any given point $p\in\mathcal{N}$, some terms of  $P_{\alpha}$ are zero and the others are not zero. In this case, without loss of generality,
assume $P_{\alpha}=0$ for $\alpha=2,\cdots, r$ and $P_{\alpha}\neq 0$ for $\alpha=r+1,\cdots,d$. Then the first $d-r$ equations in \eqref{eqsystem04} form a new system of equations
\begin{align}\label{eqsystem00}
\begin{split}
\left \{
\begin{array}{ll}
    n_{r+1}\mu_{r+1}P_{r+1}+n_{r+2}\mu_{r+2}P_{r+2}+\cdots+n_{d}\mu_{d}P_{d}=0,\\
    n_{r+1}\mu_{r+1}P_{r+1}^{2}+n_{r+2}\mu_{r+2}P_{r+2}^{2}+\cdots+n_{d}\mu_{d}P_{d}^{2}=0,\\
\quad\quad\quad\vdots\\
   n_{r+1}\mu_{r+1}P_{r+1}^{d-r}+n_{r+2}\mu_{r+2}P_{r+2}^{d-r}+\cdots+n_{d}\mu_{d}P_{d}^{d-r}=0,
\end{array}
\right.
\end{split}
\end{align}
which is a $(d-r)$-th order equation system with non-zero solutions. Thus the coefficient determinant is zero, that is $\prod_{r+1\leq\alpha\leq\beta\leq d}(P_{\alpha}-P_{\beta})=0$. Without loss of generality,
we assume that $P_{r+1}=P_{r+2}$.  Similar discussion as the above yields
$P_{r+1}=\cdots=P_{d}\neq0$.  Denote by $P:=P_{r+1}=\cdots=P_{d}$.

It follows from \eqref{eqsystem00} that $n_{r+1}\mu_{r+1}+\cdots+n_{d}\mu_{d}=0$, because of $P\neq0$. In this case, we have $n_{2}\mu_{2}+\cdots n_{r}\mu_{r}=nH$.
Therefore we deduce from Lemma \ref{lemma4} that
\begin{align}
&-6\sum_{\alpha=2}^{d}n_{\alpha}\mu_{\alpha}^{2}P_{\alpha}^{2}
+2\Big(\sum_{\alpha=2}^{d}n_{\alpha}\mu_{\alpha}^{2}P_{\alpha}\Big)\Big(\sum_{\alpha=2}^{d}n_{\alpha}P_{\alpha}+\sum_{m\in I_{1}}\Gamma_{mm}^{1}\Big)+\Big(\sum_{\alpha=2}^{d}n_{\alpha}\mu_{\alpha}^{2}\Big)^{2}=0,\nonumber
\end{align}
that is
\begin{align}\label{eqsystem180}
&2\Big(\sum_{\alpha=r+1}^{d}n_{\alpha}-3\Big)P^{2}\sum_{\alpha=r+1}^{d}n_{\alpha}\mu_{\alpha}^{2}+
2\Big(\sum_{\alpha=r+1}^{d}n_{\alpha}\mu_{\alpha}^{2}\Big)P\sum_{m\in I_{1}}\Gamma_{mm}^{1}\nonumber\\
&+\Big(\sum_{\alpha=2}^{d}n_{\alpha}\mu_{\alpha}^{2}\Big)^{2}=0,
\end{align}
where we have used $P_{\alpha}=0$ for $\alpha=2,\ldots, r$.

Because we have assumed that the multiplicity of the zero principal curvature is one, i.e. $I_1=\{1\}$, \eqref{eqsystem180} becomes
\begin{align}\label{eqsystem181}
2\Big(\sum_{\alpha=r+1}^{d}n_{\alpha}-3\Big)P^{2}\sum_{\alpha=r+1}^{d}n_{\alpha}\mu_{\alpha}^{2}
+\Big(\sum_{\alpha=2}^{d}n_{\alpha}\mu_{\alpha}^{2}\Big)^{2}=0.
\end{align}
Differentiating with respect to $e_{1}$ on both sides of equation \eqref{eqsystem181}, it follows from \eqref{F1} that
\begin{align}\label{eqsystem19}
8\Big(\sum_{\alpha=r+1}^{d}n_{\alpha}-3\Big)P^{3}\sum_{\alpha=r+1}^{d}n_{\alpha}\mu_{\alpha}^{2}
+4\Big(\sum_{\alpha=2}^{d}n_{\alpha}\mu_{\alpha}^{2}\Big)\Big(\sum_{\alpha=r+1}^{d}n_{\alpha}\mu_{\alpha}^{2}\Big)P=0.
\end{align}
Dividing \eqref{eqsystem19} by $4P$, and then subtracting  \eqref{eqsystem181}, we have
\begin{align}\label{eqsystem20}
\Big(\sum_{\alpha=2}^{d}n_{\alpha}\mu_{\alpha}^{2}\Big)\Big(\sum_{\alpha=2}^{r}n_{\alpha}\mu_{\alpha}^{2}\Big)=0.
\end{align}
Since $S=\sum_{\alpha=2}^{d}n_{\alpha}\mu_{\alpha}^{2}\neq 0$, it follows from \eqref{eqsystem20} that
$\sum_{\alpha=2}^{r}n_{\alpha}\mu_{\alpha}^{2}=0$, and hence $\mu_{\alpha}=0$ for $\alpha=2,\cdots,r$.
This is a contradiction since $\sum_{\alpha=2}^{r}n_{\alpha}\mu_{\alpha}=nH\neq0$. $\hfill\square$\\

%%%%%%%%%%%%%%%%%%%%%%%%%%%%%%%%%%%%%%%%%%
{\bf The proof of Theorem \ref{mainth4}}:
Let us restrict the case of a CMC triharmonic hypersurface $M^5$ in $\mathbb R^{6}$. We will still work on $\mathcal{N}=\{p\in M^5:\nabla S(p)\neq 0\}\neq\emptyset$ and prove Theorem \ref{mainth4} with a contradiction.

When the multiplicity of zero principal curvature is one, we can derive a contradiction from Theorem \ref{mainth3}.

By Theorem \ref{ChenGuan}, we only need to deal with the case that the multiplicity of zero principal curvature is two, i.e. $I_1=\{1, 2\}$. In this case, the principle curvatures on $M^{5}$ are respectively $0, 0, \lambda_{3}, \lambda_{4}, \lambda_{5}$ with $\lambda_{3}+ \lambda_{4}+\lambda_{5}=5H$.

According to the proof of Theorem \ref{mainth3}, we can deduce a contradiction when $P_{\alpha}\neq0$ or $P_{\alpha}=0$ for all $P_{\alpha}$.  Let us consider the remaining case that $P_{3}=0$, $P_{4}=P_{5}\neq0$. It follows from \eqref{eqsystem00} that $\lambda_4+\lambda_5=0$ and $\lambda_3=5H$. For simplicity, we denote $\lambda_{4}=-\lambda_{5}=:\mu$ and $P_{4}=P_{5}=:P$.  Then the squared norm of the second fundamental form $S$ is given by $S=2\mu^{2}+25H^2$. Since $e_{2}(S)=0$, we have $e_{2}(\mu)=0$ and hence  $\Gamma_{44}^{2}=\Gamma_{55}^{2}=0$. In this case, we deduce from \eqref{F2} that $e_2(P)=\Gamma_{44}^2P-\Gamma_{44}^2\Gamma_{22}^1=0$. Moreover,
$e_{1}(S)=4\mu e_{1}(\mu)=4\mu^{2}P$ and hence $e_{2}e_{1}(S)=0$. Since $e_1(S)\neq0$, $e_i(S)=0$ for $i\geq2$ and $\Gamma_{21}^1=0$, we find
\begin{align}
0=e_{1}e_{2}(S)-e_{2}e_{1}(S)=[e_{1}, e_{2}](S)=(\nabla_{e_{1}}e_{2}-\nabla_{e_{2}}e_{1})S=\Gamma_{12}^{1}e_{1}(S),\nonumber
\end{align}
which means that $\Gamma_{12}^{1}=-\Gamma_{11}^{2}=0$. Then \eqref{F1} turns into
\begin{align}\label{L0}
e_{1}(P)=P^{2}.
\end{align}

One the other hand, taking into account the Gauss equation and \eqref{gauss}, from Lemma \ref{lemma1} we obtain
\begin{align}
0=R_{1212}=&e_{1}(\Gamma_{22}^{1})-e_{2}(\Gamma_{12}^{1})+\sum_{m}\Big(\Gamma_{22}^{m}\Gamma_{1m}^{1}-
\Gamma_{12}^{m}\Gamma_{2m}^{1}-(\Gamma_{12}^{m}-\Gamma_{21}^{m})\Gamma_{m2}^{1}\Big)\nonumber\\
=&e_{1}(\Gamma_{22}^{1})-(\Gamma_{22}^{1})^{2}. \label{L1}
\end{align}
Based on the above discussion, \eqref{eqsystem180} becomes
\begin{align}\label{G1}
-4\mu^{2}P^{2}+4\mu^{2}P\Gamma_{22}^{1}+(2\mu^{2}+25H^2)^{2}=0,
\end{align}
and hence
\begin{align}\label{G2}
\Gamma_{22}^{1}=\frac{4\mu^{2}P^{2}-(2\mu^{2}+25H^2)^{2}}{4\mu^{2}P}.
\end{align}
Noting $e_1(\mu)=\mu P$ and differentiating \eqref{G1} with respect to $e_{1}$, from \eqref{L0} and \eqref{L1} we have
\begin{align}\label{G3}
-4P^{2}+3 P\Gamma_{22}^{1}+(\Gamma_{22}^{1})^{2}+2(2\mu^{2}+25H^2)=0.
\end{align}
Substituting \eqref{G2} into \eqref{G3} and eliminating the terms of $\Gamma_{22}^{1}$  one gets
\begin{align}
&16\mu^4P^2\Big(2(2\mu^{2}+25H^2)-4P^{2}\Big)+12\mu^2P^2\Big(4\mu^{2}P^{2}-(2\mu^{2}+25H^2)^{2}\Big)\nonumber\\
&+\Big(4\mu^{2}P^{2}-(2\mu^{2}+25H^2)^{2}\Big)^2=0,\nonumber
\end{align}
that is
\begin{align} \label{G4}
32\mu^4P^2(2\mu^{2}+25H^2)-20\mu^2P^2(2\mu^{2}+25H^2)^2+(2\mu^{2}+25H^2)^{4}=0.
\end{align}
Because $2\mu^{2}+25H^2\neq0$, \eqref{G4} becomes
\begin{align} \label{G5}
32\mu^4P^2-20\mu^2P^2(2\mu^{2}+25H^2)+(2\mu^{2}+25H^2)^{3}=0.
\end{align}
Differentiating with respect to $e_{1}$ on  \eqref{G5}, from \eqref{L0} and \eqref{L1} we have
\begin{align} \label{G6}
-12\mu^2P^2-500H^2P^2+3(2\mu^{2}+25H^2)^{2}=0.
\end{align}
Differentiating with respect to $e_{1}$ on \eqref{G6} yields
\begin{align} \label{G7}
-12\mu^2P^2-250H^2P^2+6\mu^2(2\mu^{2}+25H^2)=0.
\end{align}
Eliminating the terms of $P^2$ between \eqref{G6} and \eqref{G7} gives
\begin{align}
2\mu^2(12\mu^2+500H^2)=(2\mu^{2}+25H^2)(12\mu^2+250H^2),\nonumber
\end{align}
that is
\begin{align} \label{G8}
4\mu^2H^2-125H^4=0,
\end{align}
which implies that $\mu$ is constant and $S=2\mu^2+25H^2$ is constant as well. This contradicts $\nabla S(p)\neq 0$.

In summary, we conclude that $\mathcal{N}=\{p\in M:\nabla S(p)\neq 0\}=\emptyset$ and according to \eqref{triharmonic condition1} we know that $M^n$ is minimal. This completes the proof of Theorem \ref{mainth4}. $\hfill\square$

\medskip

\medskip\noindent
{\bf Acknowledgement:} {The authors are supported by the NSFC (No.11801246) and  Liaoning
Provincial Education Department Project (No.LJKMZ20221561) }

%%%%%%%%%%%%%%%%%%%%%%%%%%%%%%%%%%%%%%%%%%%%%%%%%%%%%%%%%%%%%%

%%%%%%%%%%%%%%%%%%%%%%%%%%%%%%%%%%%%%%%%%%%%%%%%%%%%%%%%%%%%%%

%\bibliographystyle{amsplain}

\end{document}